\documentclass[12pt]{amsart}
\usepackage[colorlinks=true,linkcolor=blue]{hyperref}
\usepackage{amsmath}
\usepackage{amssymb}
\usepackage{amsthm}
\usepackage{graphicx}

\theoremstyle{plain} 
\newtheorem{theorem}{Theorem}[section]
\newtheorem{proposition}[theorem]{Proposition}
\newtheorem{lemma}[theorem]{Lemma}

\newtheorem{conjecture}[theorem]{Conjecture}

\theoremstyle{remark}

\theoremstyle{definition}

\DeclareMathOperator{\Gal}{Gal}

\DeclareMathOperator{\charact}{char}

\DeclareMathOperator{\Aut}{Aut}

\newcommand{\fp}{ {\mathfrak p} }

\newcommand{\fq}{ {\mathfrak q} }

\newcommand{\fP}{ {\mathfrak P} }
\newcommand{\fQ}{ {\mathfrak Q} }

\newcommand{\cO}{ {\mathcal O} }

\newcommand{\bQ} { {\mathbb Q}}

\newcommand{\QQ} { {\mathbb Q}} 
\newcommand{\ZZ} { {\mathbb Z}}

\usepackage{tikz}

\begin{document}

\title{Odoni's conjecture for number fields}

\author[Benedetto]{Robert L. Benedetto}
\address[Benedetto]{Amherst College \\ Amherst, MA}
\email[Benedetto]{rlbenedetto@amherst.edu}

\author[Juul]{Jamie Juul}
\address[Juul]{Amherst College \\ Amherst, MA}
\email[Juul]{jamie.l.rahr@gmail.com}

\begin{abstract}
Let $K$ be a number field, and let $d\geq 2$.
A conjecture of Odoni (stated more generally for characteristic zero Hilbertian fields $K$) posits
that there is a monic polynomial $f\in K[x]$ of degree $d$, and a point $x_0\in K$,
such that for every $n\geq 0$, the so-called arboreal Galois group $\Gal(K(f^{-n}(x_0))/K)$ is an
$n$-fold wreath product of the symmetric group $S_d$.
In this paper, we prove Odoni's conjecture when $d$ is even and $K$ is an arbitrary
number field, and also when both $d$ and $[K:\bQ]$ are odd.
\end{abstract}



\maketitle

\section{Introduction}

Let $F$ be a field, let $f(x)\in F[x]$ be a polynomial of degree $d\geq 2$, and let $x_0\in F$.
For each $n\geq 0$,
denote by $f^n$ the $n$-iterate $f \circ f\circ \dots \circ f$ of $f$,
and consider the set $f^{-n}(x_0)=\{\alpha\in \bar{F}\;|\; f^n(\alpha)=x_0\}$
of $n$-th preimages of $x_0$. If $f^n-x_0$ is separable for all $n$,
then $f^{-n}(x_0)$ has exactly $d^n$ elements for each $n$,
and $F(f^{-n}(x_0))$ is a Galois extension of $F$.

In \cite{odoni}, Odoni showed that $\Gal(F(f^{-n}(x_0))/F)$
is isomorphic to a subgroup of $[S_d]^n$,
the $n$-fold wreath product of the symmetric group $S_d$ with itself.
He also showed that if $\charact F=0$ and $E = F(s_{d-1}, \dots, s_0)$,
then the generic monic polynomial
$G(x)=x^d+s_{d-1}x^{d-1}+\dots+s_0\in E[x]$
defined over the function field $E$ satisfies $\Gal(E(G^{-n}(0))/E)\cong [S_d]^n$.
In \cite{JO}, the second author showed that this result also holds for fields of characteristic $p$,
except in the case $p=d=2$. It follows from Hilbert's Irreducibility Theorem that if $F=\bQ$, or more generally if $F$ is any Hilbertian field,
then for any fixed $n\geq 0$, there are infinitely many polynomials $f(x)\in F(x)$
for which $\Gal(F(f^{-n}(x_0))/F)\cong [S_d]^n$.
However, it does not follow immediately that there are any polynomials $f(x)\in F[x]$
for which this isomorphism holds for \emph{all} $n\geq 0$.
Based on his results, Odoni proposed the following conjecture.

\begin{conjecture}[Odoni, Conjecture 7.5 \cite{odoni}]
\label{conj:odoni}
For any Hilbertian field $F$ of characteristic $0$ and any $d\geq 2$,
there is a monic polynomial $f(x)\in F[x]$ of degree $d$ such that
$\Gal(F(f^{-n}(0))/F)\cong [S_d]^n$ for all $n\geq 0$.
\end{conjecture}

For any point $x_0\in F$, if we set $g(x)=f(x+x_0)-x_0\in F[x]$,
then the fields $F(f^{-n}(x_0))$ and $F(g^{-n}(0))$ coincide. Thus, it
is equivalent to phrase Odoni's conjecture in terms of the preimages $f^{-n}(x_0)$
of an arbitrary $F$-rational point $x_0$ instead of $0$.

The Galois groups $\Gal(F(f^{-n}(x_0)/F)$ can be better understood through the framework of arboreal Galois representations \cite{BJ}. It is not hard to see that $[S_d]^n\cong \Aut(T_{d,n})$ where $T_{d,n}$ is a $d$-ary rooted tree with $n$ levels.
We define an embedding $\Gal(F(f^{-n}(x_0))/F)\rightarrow \Aut(T_{d,n})$
by assigning each element of $\bigsqcup_{i=1}^n f^{-i}(x_0)$ to a vertex of the tree as follows:
$x_0$ is the root of the tree, and the points of $f^{-i}(x_0)$ are the vertices at the $i$-th
level of the tree, with an edge connecting $\alpha\in f^{-i}(x_0)$ to $\beta\in f^{-i-1}(x_0)$
if $f(\beta)=\alpha$.
		
Jones \cite{Jones3} stated a version Odoni's Conjecture in the case that $F=\QQ$
by further specifying that $f$ should have coefficients in $\ZZ$. In this paper, however, we consider
the original version of Conjecture~\ref{conj:odoni},
where $f$ is allowed to have non-integral coefficients.
		
Conjecture~\ref{conj:odoni} has already been proven in many cases.
Odoni himself proved that $\Gal(\bQ(f^{-n}(0))/\bQ)\cong [S_2]^n$ for all $n\geq 0$
when $f(x)=x^2-x+1$, proving the conjecture for  $F=\bQ$ and $d=2$ \cite{odoni2}.
Stoll \cite{Stoll} produced infinitely many such polynomials for $F=\bQ$ and $d=2$.
In 2017, Looper showed that Odoni's conjecture holds for $F=\bQ$ and $d=p$ a prime \cite{Looper}. 

In this paper, we prove Odoni's conjecture for even $d\geq 2$ over any number field,
as well as for odd $d\geq 3$ over any number field $K$ not containing $\bQ(\sqrt{d},\sqrt{d-2})$.
(In particular, we prove the conjecture when $d$ and $[K:\bQ]$ are both odd.)

\begin{theorem}\label{thm:maintheorem}
Let $d\geq 2$, and let $K$ be a number field. Suppose either that $d$ is even or that
$d$ and $d-2$ are not both squares in $K$.
Then there is a monic polynomial $f(x)\in K[x]$ of degree $d$
and a rational point $x_0\in K$ such that for all $n\geq 0$,
\[ \Gal \Big( K\big( f^{-n}(x_0) \big) / K \Big) \cong [S_d]^n .\]
\end{theorem}


The proof, which builds on Looper's techniques, proceeds by induction on $n$,
and involves studying the primes ramifying in $K(f^{-n}(x_0))$. In particular,
to help generate the full group $S_d$ when $d$ is not necessarily prime,
we introduce a positive integer $m<d$ and an auxiliary prime $K$ that ramifies
to degree $m^n$ in $K(\alpha)$, for any $\alpha\in f^{-n}(x_0)$.

Recently, Borys Kadets \cite{Kadets} and Joel Specter \cite{Specter}
have announced proofs of similar theorems, and using similar extensions of Looper's techniques.
Kadets proves Odoni's conjecture over $\bQ$ for polynomials of even degree $d\geq 20$.
Specter proves Odoni's conjecture for algebraic extensions of $\bQ$ which are unramified outside of infinitely many primes.
Our work, which we announced at the 2018 Joint Math Meetings in San Diego
(see \texttt{https://rlbenedetto.people.amherst.edu/talks/sandiego18.pdf}),
was done simultaneously and independently from these projects. 

The outline of the paper is as follows. In Section \ref{section:discriminant} we give preliminary results on discriminant formulas and ramification, as well as a useful group theory lemma,
Lemma~\ref{lem:genSd}.
In Section~\ref{section:sufficientconditions} we prove sufficient conditions for
$\Gal(K(f^{-n}(x_0))/K)$ to be isomorphic to $[S_d]^n$.
Finaly, we prove Theorem \ref{thm:maintheorem} for even $d\geq 2$
in Section~\ref{section:even}, and for odd $d\geq 3$ in Section~\ref{section:odd}.

\section{Ramification and the Discriminant}\label{section:discriminant}

We begin with the following result on the discriminant of a field generated
by a root of a trinomial.

\begin{lemma}
\label{lemma.disc}
Let $K$ be a number field with ring of integers $\cO_K$, let $d>m\geq 1$ with $(m,d)=1$,
let $A,B,C\in K$ with $A\neq 0$, and let
$S$ be a finite set of primes of $\cO_K$ including all archimedean primes and
all primes at which any of $A,B,C$ have negative valuation.

Suppose that
$g(x) = Ax^d+Bx^{m}+C\in\cO_{K,S}[x]$ is irreducible over $K$.
Let $L=K(\theta)$, where $\theta$ is a root of $g(x)$. 
Then the discriminant $\Delta(\cO_{L,S}/\cO_{K,S})$ satisfies
$\Delta(g)=k^2\Delta(\cO_{L,S}/\cO_{K, S})$, where $k\in \cO_{K,S}$, and
\[\Delta(g)=(-1)^{d(d-1)/2}A^{d-m-1}C^{m-1}
\left[(-1)^{d-1} m^m (d-m)^{d-m} B^d + d^d A^{m} C^{d-m}\right] \] 
is the discriminant of the polynomial $g$.
\end{lemma}

We will also make use of the following discriminant formula:
\begin{equation}
\label{eq:iterdisc}
\Delta(f^{n+1}(x)-t)=\tilde{A}^{d^n}[\Delta(f^n(x)-t)]^d\prod_{f'(r)=0}(f^{n+1}(r)-t)^{m_r},
\end{equation}
where $f$ is a polynomial of degree $d$ and lead coefficient $A$,
where $\tilde{A}=(-1)^{d(d-1)/2}d^d A^{d-1}$,
and where $m_r$ is the multiplicity of $r$ as a root of $f'(x)$.
See \cite[Proposition 3.2]{AHM}.

\begin{proof}[Sketch of Proof of Lemma~\ref{lemma.disc}]
This is a standard result,
using the fact that any prime $\fp\nmid A$ of $\cO_{K,S}$ ramifying in $\cO_{L,S}$ divides
$\Delta(\cO_{L,S}/\cO_{K,S})$.
See, for example, Lemma~7.2, Theorem~7.3, and Theorem~7.6 of \cite{Jan}.

To prove the formula for $\Delta(g)$, we apply formula~\eqref{eq:iterdisc}
with $f=g$, $n=0$, and $t=0$. More precisely, $x=0$ is a critical point of $g$
of multiplicity $m-1$, and we have $g(0)^{m-1}=C^{m-1}$. The other critical points
are $\zeta^j \eta$ for $1\leq j \leq d-m$, where $\eta$ is a $(d-m)$-th root of
$-mB/(dA)$, and where $\zeta$ is a primitive $(d-m)$-th root of unity. Thus,
\begin{align}
\label{eq:gprod}
\prod_{j=1}^{d-m} g\left( \zeta^j \eta \right) &=
\prod_{j=1}^{d-m} \left( C - \left( \frac{m-d}{d} \right) B \zeta^{jm}\eta^m \right)
\notag \\
&= C^{d-m} - \left[ \left( \frac{m-d}{d} \right)^{d-m} B^{d-m} \left( -\frac{mB}{dA} \right)^m \right]
\\
&= d^{-d} A^{-m} \left[ d^d A^m C^{d-m} + (-1)^{d-1} (d-m)^{d-m} m^m B^d \right],
\notag
\end{align}
where we have used the fact that $d$ and $(d-m)$ are relatively prime in the second equality,
to deduce that
\[ \{\zeta^{jm} : 1\leq j\leq d-m\}=\{\zeta^{j} : 1\leq j\leq d-m\} .\]
Multiplying by $g(0)^{m-1}$ and $(-1)^{d(d-1)/2}d^dA^{d-1}$ as in formula~\eqref{eq:iterdisc},
the desired formula for $\Delta(g)$ follows immediately.
\end{proof}


\begin{lemma}
\label{lemma.transp}
Let $A,B,C,d,m,g,K,L$ be as in Lemma~\ref{lemma.disc}. If a prime $\fp\nmid ABC$ of $\cO_{K,S}$
ramifies in $\cO_{L,S}$, and if $\fq$ is a prime of the Galois closure of $L$ over $K$,
then $\fp\nmid md(d-m)$,
and the ramification group $I(\fq|\fp)$ is generated by a single transposition
of the roots of $g$.
\end{lemma}

\begin{proof} 
Let $\fp\nmid ABC$ be a prime of $\cO_{K,S}$ which ramifies in $\cO_{L,S}$.
Then $\fp\mid \Delta(\cO_{L,S}/\cO_{K,S})$, and hence
by Lemma \ref{lemma.disc}, we also have $\fp\nmid md(d-m)$, since $(d,m)=1$.

Because $\fp$ ramifies, $g(x)$ must have at least one multiple root modulo $\fp$.
On the other hand, if $\eta$ is a mod-$\fp$ root of multiplicity $\ell >2$,
then $\eta$ is also at least a double root of the derivative
$g'(x)\equiv dAx^{d-1}+mBx^{m-1} \pmod{\fp}$.
However, since $\fp\nmid ABCdm(d-m)$, this cannot be the case unless $\eta\equiv 0 \pmod{\fp}$;
but then $\eta$ would not have been a root of $g$ itself, since $\fp\nmid C$.
Therefore each root of $g(x) \pmod{\fp}$ has at multiplicity at most two.

Now suppose $\eta$ and $\xi$ are both double roots of $g(x)\pmod{\fp}$.
Then both $\eta$ and $\xi$ are nonzero simple roots of $g'(x)\pmod{\fp}$,
and hence
\[ \eta^{d-m}\equiv \xi^{d-m}\equiv \frac{-mB}{dA}\pmod{\fp} . \]
Thus, $\eta\equiv\zeta\xi\pmod{\fp}$, where $\zeta$ is a $(d-m)$-th root of unity.
If $\zeta\not\equiv 1 \pmod{\fp}$, then $\zeta^m\not\equiv 1 \pmod{\fp}$,
since $(m,d-m)=1$. Therefore,
\[ g(\eta) = C+\big( g(\eta) - C \big) \equiv C+\zeta^m \big( g(\xi) - C \big)
\equiv (1-\zeta^m) C \not\equiv 0 \pmod{\fp}, \]
a contradiction. Hence, we must have $\eta\equiv\xi \pmod{\fp}$.

The two previous paragraphs together yield that $g$ has exactly one multiple root
modulo $\fp$, and it is a double root.
That is,
\[ g(x) \equiv A(x-\eta)^2 g_1(x) \pmod{\fp}, \]
where $g_1(x)\in \cO_{K,S}/\fp [x]$ is a separable polynomial with
$g_1(\eta)\not\equiv 0 \pmod{\fp}$.
Since $g$ is irreducible over $K$, it follows that
$I(\fq|\fp)$ is generated by a single transposition, for any $\fq$ lying above $\fp$.
\end{proof}

%

\begin{lemma}\label{lemma.minertia}
Let $d>m\geq 2$, let $b, x_0\in K$,
and suppose there is a prime $\fp$ of $\cO_K$ such that
$\fp\nmid (d-m)$, and $v_{\fp}(b) < \min\{v_{\fp}(x_0),0\}$,
with $(d-m) | v_{\fp}(b)$ and $\gcd(m, v_{\fp}(x_0/b))=1$.
Let $f(x) = x^d - b x^{m}$, let $n\geq 0$, and let $\alpha \in f^{-n}(x_0)$.
Suppose that $f^n(x)-x_0\in K[x]$ is irreducible over $K$.
Then there is a prime $\fP$ of $K(\alpha)$ lying above $\fp$,
and a prime $\fQ$ of $K(f^{-1}(\alpha))$, such that
\begin{itemize}
\item $\fP$ has ramification index $m^n$ over $\fp$,
\item $m^n v_{\fP}(\alpha/b)$ is a positive integer relatively prime to $m$,
where $v_{\fP}$ is the $\fP$-adic valuation on $K(\alpha)$ extending $v_{\fp}$,
\item $\fQ$ lies above $\fP$, and
\item the ramification group $I(\fQ|\fP)$ acts transitively on $m$ roots of $f(x)-\alpha$
and fixes the other $d-m$ roots.
\end{itemize}
\end{lemma}

\begin{proof}
\textbf{Step~1}. We will prove the first two bullet points by induction on $n\geq 0$.
For $n=0$, we have $\alpha=x_0$; choosing $\fP=\fp$, both points hold trivially.

Assuming they hold for $n-1$, let $\beta=f(\alpha)\in f^{-(n-1)}(x_0)$. By our
hypothesis that $f^n(x)-x_0$ is
irreducible over $K$, the previous iterate $f^{(n-1)}(x)-x_0$ must also be irreducible.
By our inductive hypothesis,
there is a prime $\fP'$ of $K(\beta)$ with ramification index $m^{n-1}$ over $\fp$,
such that  $m^{n-1} v_{\fP'}(\beta/b)$ is a positive integer relatively prime to $m$.

\begin{figure}
\scalebox{1}{\includegraphics{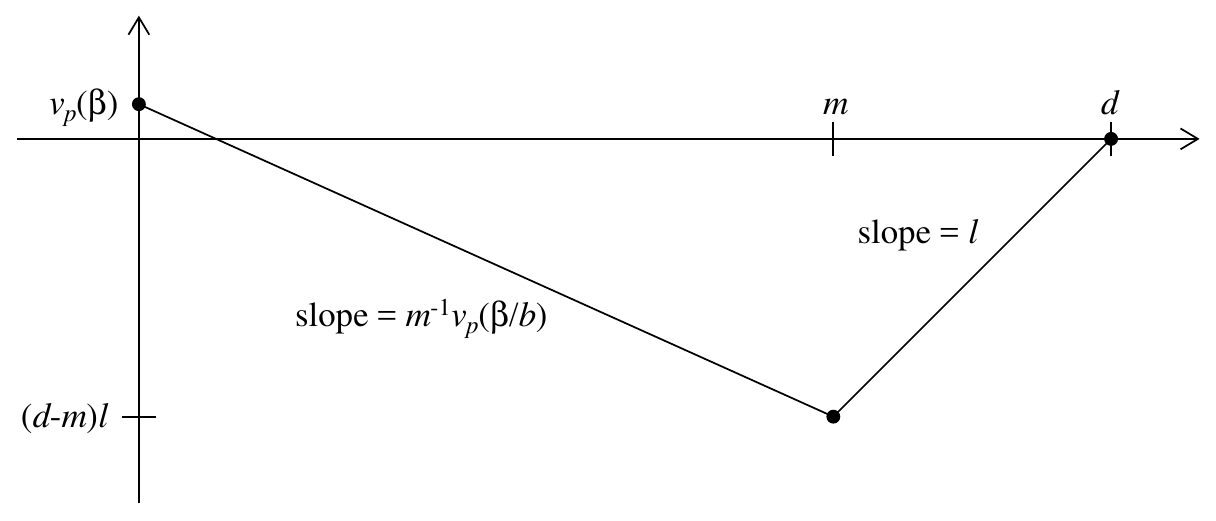}}
\caption{The Newton polygon for \mbox{$f(x)-\beta$} in Lemma~\ref{lemma.minertia}.}
\label{fig:newtpoly}
\end{figure}

Thus, the Newton polygon of $f(x)-\beta = x^d -bx^m - \beta$ with respect to
$v_{\fP'}$ consists of two segments: one of length $m$ and slope
$-m^{-1} v_{\fP'}(\beta/b)<0$,
and one of length $d-m$ and slope $-v_{\fp}(b)/(d-m)$, which is a positive integer;
see Figure~\ref{fig:newtpoly}.
That is, $f(x)-\beta$ factors over the local field $K(\beta)_{\fP'}$
as $f=gh$, where $g,h\in K(\beta)_{\fP'}[x]$,
with $\deg(g)=m$ and $\deg(h)=d-m$.
Moreover, the Newton polygon of each of $g$ and $h$ consists of a single segment,
of slopes $-m^{-1} v_{\fP'}(\beta/b)$ and $-v_{\fp}(b)/(d-m)$, respectively.
Since $f^n-x_0$ is irreducible, then considering the Galois extension
$K_n:=K(f^{-n}(x_0))$,
we may apply an appropriate $\sigma\in\Gal(K_n/K)$ to assume that $\alpha$ is a root of $g$.

Because $v_{\fP'}(\beta/b)=N/m^{n-1}$ for some positive integer $N$ relatively prime to $m$,
and because the Newton polygon of $g$ consists of a single segment of slope $-N/m^n$,
it follows that $K(\alpha)$ has a prime $\fP$ of ramification index $m$ over $\fP'$,
and hence of index $m^n$ over $\fp$, proving the first bullet point.
Letting $v_{\fP}$ denote the $\fP$-adic valuation on $K(\alpha)$ extending $v_{\fP'}$,
we have $v_{\fP}(\alpha)=N/m^n >0$. Since $v_{\fP}(b)=v_{\fp}(b)$ is a negative integer,
it follows that $m^n v_{\fP}(\alpha/b)$ is a positive integer relatively prime to $m$,
proving the second bullet point.

\textbf{Step 2}. Let $L:=K(f^{-1}(\alpha))$,
let $\fP$ be a prime of $K(\alpha)$ satisfying the first two bullet points,
and let $v_{\fP}$ be the $\fP$-adic valuation on $K(\alpha)$ extending $v_{\fp}$.
As in Step~1, we may factor $f(x)-\alpha=g(x) h(x)$,
where the Newton polygons of $g,h\in K(\alpha)_{\fP}[x]$
each consist of a single segment,
of length $m$ and slope $-m^{-1} v_{\fP}(\alpha/b)$ for $g$,
and of length $d-m$ and slope $\ell:=-v_{\fp}(b)/(d-m)$ for $h$.
In fact, if $\pi\in\cO_K$ is a uniformizer for $\fp$, then the polynomial
\[ F(x) := \pi^{d\ell} \Big( f\big(\pi^{-\ell} x\big) - \alpha \Big) = x^d - \pi^{(d-m)\ell} b x^m - \pi^{d\ell}\alpha
\in K(\alpha)[x] \]
has $K(\alpha)_{\fP}$-integral coefficients, with $F\equiv x^m(x^{d-m}- c) \pmod{\fP}$,
where $c\not\equiv 0 \pmod{\fP}$.
Therefore, the polynomial
\[ H(x) := \pi^{(d-m)\ell} h\big( \pi^{-\ell} x\big) \in K(\alpha)_{\fP}[x] \]
also has $K(\alpha)_{\fP}$-integral coefficients,
and $H\equiv x^{d-m} - c\pmod{\fP}$.
Since $\fp\nmid (d-m)$, the splitting field of $H$, and hence of $h$, is unramified over $\fP$.

Let $\fQ$ be any prime of $L$ lying over $\fP$, and let $\gamma_1,\ldots,\gamma_d$
be the roots of $f(x)-\alpha$ in the local field $L_{\fQ}$.
By the factorization $f-\alpha=gh$ of the previous paragraph, $d-m$ of the roots
(without loss, $\gamma_{m+1},\ldots,\gamma_d$) are roots of $h(x)$,
and hence they lie in an unramified extension $L'_{\fQ}$
of $K(\alpha)_{\fP}$ contained in $L_{\fQ}$.
On the other hand, the remaining roots  $\gamma_1,\ldots,\gamma_m$ are roots of $g$,
which is totally ramified over $K(\alpha)_{\fP}$
and hence irreducible over $L'_{\fQ}$.

The decomposition group $D(\fQ|\fP)$ is canonically isomorphic to the Galois group
$\Gal(L_{\fQ}/K(\alpha)_{\fP})$, and the
inertia group $I(\fQ|\fP)$ is canonically isomorphic to the Galois group $\Gal(L_{\fQ}/L'_{\fQ})$.
Thus, $I(\fQ|\fP)$ acts transitively on the roots of $g$, while fixing the roots of $h$,
as desired.
\end{proof}

\begin{lemma}
\label{lem:genSd}
Let $d\geq 3$, let $m$ be an integer relatively prime to $d$ with $d/2<m<d$,
and let $G\subseteq S_d$ be a subgroup that
\begin{itemize}
\item contains a transposition,
\item acts transitively on $\{1,2,\ldots, d\}$, and
\item has a subgroup $H$ that acts trivially $\{m+1,m+2,\ldots,d\}$
and transitively on $\{1,2,\ldots, m\}$.
\end{itemize}
Then $G=S_d$.
\end{lemma}

\begin{proof}
\textbf{Step~1}. Define a relation $\sim$ on $\{1,\ldots, d\}$ by
$x\sim y$ if either $x=y$, or the transposition $(x,y)$ is an element of $G$.
Clearly $\sim$ is reflexive and symmetric. It is also transitive, because if
$x,y,z\in\{1,\ldots, d\}$ are distinct with $x\sim y$ and $y\sim z$, then
\[ (x,z) = (x,y)(y,z)(x,y)\in G, \quad \text{so } x\sim z.\]
(We also clearly have $x\sim z$ if any two of $x,y,z$ coincide.)
Thus, $\sim$ is an equivalence relation.

\textbf{Step~2}. We claim that each equivalence class of $\sim$ has the same size.
To see this, given $x,y \in \{1,\ldots, d\}$, denote by $[x]$ and $[y]$ the $\sim$-equivalence
classes of $x$ and $y$, respectively. Since $G$ acts transitively, there is some $\sigma\in G$
such that $\sigma(x)=y$. Then $\sigma$ maps $[x]$ into $[y]$, because for any $t\in [x]$,
we have $(x,t)\in G$, and hence
\[ \big( y, \sigma(t) \big) = \big( \sigma(x) , \sigma(t) \big) =
\sigma \circ (x,t) \circ \sigma^{-1} \in G, \]
so that $\sigma(t)\in [y]$. Similarly, $\sigma^{-1}$ maps $[y]$ into $[x]$.
Thus, $\sigma:[x]\to [y]$ is an invertible function, proving the claim.

\textbf{Step~3}. Let $j$ denote the common size of each equivalence class of $\sim$.
Then $j\geq 2$, since $G$ contains a transposition. Also, $j|d$ by Step~2, and 
because $\gcd(m,d)=1$, it follows that $j\nmid m$.
Thus, there must be some $x\in\{1,\ldots, m\}$ such that $[x]\not\subseteq \{1,\ldots, m\}$.
That is, there is some $y\in\{m+1,\ldots, d\}$ such that $(x,y)\in G$.

We claim that in fact, $\{1,\ldots, m\} \subseteq [y]$.
Indeed, for any $t\in \{1,\ldots, m\}$, there is some $\tau\in H$ such that $\tau(x)=t$;
and since $H$ acts trivially on $\{m+1,\ldots, d\}$, we also have $\tau(y)=y$.
Thus, $(y,t) = \tau\circ (x,y) \circ\tau^{-1}\in G$. That is, $t\sim y$, proving our claim.

This claim immediately implies that $j\geq m > d/2$.
Since $j|d$, we must have $j=d$. Hence, the whole set $\{1,\ldots,d\}$ is a single
equivalence class. That is, every transposition belongs to $G$;
therefore, $G=S_d$.
\end{proof}

\section{Sufficient conditions for large arboreal Galois groups}\label{section:sufficientconditions}

Our main tools for proving that certain arboreal Galois groups are as large as possible
are the following theorems.
We will apply the first to polynomials of degree $d\geq 4$, and the second to degrees $d=2,3$.

\begin{theorem}
\label{thm:biggalois}
Fix integers $d>m\geq 2$ with $(m,d)=1$ and $m>d/2$.
Let $b,x_0\in K$, where we can write $x_0=s/t$ and $b=u/w$
with $s,t,u,w\in \cO_K$ and $(s,t)=(u,w)=1$.
Suppose that the polynomial
\[ f(x) = x^d - bx^{m} \]
satisfies the following properties.
\begin{enumerate}
	\item there is a prime $\fp_1$ of $\cO_K$ with $v_{\fp_1}(b)\geq 1$ and $v_{\fp_1}(x_0)=1$;
	\item there is a prime $\fp_2$ of $\cO_K$ such that
	\begin{enumerate}
		\item $\fp_2\nmid (d-m)$,
		\item $v_{\fp_2}(b)<\min\{v_{\fp_2}(x_0),0\}$,	
		\item $(d-m) | v_{\fp_2}(b)$,
		\item $\gcd(m,v_{\fp_2}(x_0/b))=1$;
	\end{enumerate}
	\item for each $n\geq 1$, there is a prime $\fp\nmid stuwd(d-m)$ of $\cO_K$
	such that $v_\fp(\Delta(f^n(x) - x_0))>0$ is odd, and such that
	$v_\fp(\Delta(f^{\ell}(x) - x_0))=0$ for all $0\leq \ell<n$.
\end{enumerate}
Then for all $n\geq 0$,
\[ \Gal\Big(K\big(f^{-n}(x_0)\big)/K\Big)\cong [S_d]^n . \]
\end{theorem}

For quadratic and cubic polynomials, conditions~(1) and~(3) of Theorem~\ref{thm:biggalois}
suffice, as follows.

\begin{theorem}
\label{thm:biggalois23}
Let $d=2$ or $d=3$.
Let $b,x_0\in K$, where we can write $x_0=s/t$ and $b=u/w$
with $s,t,u,w\in \cO_K$ and $(s,t)=(u,w)=1$.
Suppose that the polynomial
\[ f(x) = x^d - bx \]
satisfies properties~\emph{(1)} and~\emph{(3)} of Theorem~\ref{thm:biggalois}, with $m=1$.
Then for all $n\geq 0$,
\[ \Gal\Big(K\big(f^{-n}(x_0)\big)/K\Big)\cong [S_d]^n . \]
\end{theorem}

The proofs of Theorems~\ref{thm:biggalois} and~\ref{thm:biggalois23}
rely on the following two results.

\begin{proposition}
\label{prop:galSd}
Let $d$, $m$,  $x_0=s/t$, $b=u/w$, and $f$ be as in
Theorem~\ref{thm:biggalois} or Theorem~\ref{thm:biggalois23}.
Then for any $n\geq 1$ and $\alpha\in f^{-(n-1)}(x_0)$,
$\Gal(K(f^{-1}(\alpha))/K(\alpha))\cong S_d$.
\end{proposition}

\begin{proof}
For convenience of notation, let $G:=\Gal(K(f^{-1}(\alpha))/K(\alpha))$,
and let $\beta_1,\ldots,\beta_d$ be the roots of $f(x)-\alpha$.

Observe that $f(x)\equiv x^d \pmod{\fp_1[x]}$,
where $\fp_1$ is the prime described in condition~(1) of Theorem~\ref{thm:biggalois}.
Thus, $f^{n}(x)\equiv x^{d^{n}} \pmod{\fp_1[x]}$,
and in addition, the constant term of $f^{n}(x)$ is trivial. Hence,
$f^{n}(x)-x_0$ is Eisenstein at $\fp_1$ and therefore irreducible over $K$,
of degree $d^{n}$.
	
Further, since $f^n(x)-x_0$ is irreducible over $K$,
$f(x)-\alpha$ is irreducible over $K(\alpha)$ by Capelli's Lemma. 
In particular, $G$ acts transitively on $\{\beta_1,\ldots,\beta_d\}$.

Let $\fp$ be the prime described in condition~(3) of Theorem~\ref{thm:biggalois}
for $\Delta(f^{n}(x)-x_0)$.
By equation~\eqref{eq:iterdisc} and the fact that $\fp\nmid stuwd$,
we see that $\fp$ must divide
\[\prod_{i=1}^{d-m} \big( f^{n}(\zeta^i\eta)-x_0\big) =
\prod_{\gamma \in f^{-(n-1)}(x_0)}\left(\prod_{i=1}^{d-m} \big( f(\zeta^i\eta)-\gamma \big) \right)\]
to an odd power, where $\eta$ is a nonzero critical point of $f(x)$,
and $\zeta$ is a primitive $(d-m)$-root of unity.
Hence, there is some prime $\fq'$ of $K_n:=K(f^{-n}(x_0))$ lying above $\fp$,
along with some Galois conjugate $\alpha'$ of $\alpha$, such that $\fq'$ divides
$\prod_{i=1}^{d-m}(f(\zeta^i\eta)-\alpha' )$ to an odd power.
Since $\alpha'$ and $\alpha$ are conjugates,
there must also be a prime $\fq$ lying above $\fp$
dividing  $\prod_{i=1}^{d-m}(f(\zeta^i\eta)-\alpha)$ to an odd power.
Finally, restricting $\fq$ to $K(\alpha)$,
we see that there is a prime $\fP$ of $K(\alpha)$ lying above $\fp$
that divides $\Delta(f(x)-\alpha)$ to an odd power.
Applying Lemma~\ref{lemma.disc} to $\Delta(f(x)-\alpha)$,
the prime $\fP$ must ramify in $K(f^{-1}(\alpha))$.
By Lemma~\ref{lemma.transp}, the corresponding inertia subgroup in $G$
must be generated by a single transposition of the roots $\{\beta_1,\ldots,\beta_d\}$.
	

Thus, $G=\Gal(K(f^{-1}(\alpha))/K(\alpha))$ is a subgroup of $S_d$
that acts transitively on $\{\beta_1,\ldots,\beta_d\}$
and that also contains a transposition.
For $d=2$ or $d=3$, it follows that $G\cong S_d$,
proving the desired result
under the hypotheses of Theorem~\ref{thm:biggalois23}.

For the remainder of the proof,
assume the hypotheses of Theorem~\ref{thm:biggalois}.
Let $\fp_2$ be the prime described in condition~(2) of that Theorem.
Then Lemma~\ref{lemma.minertia} applied to $\fp_2$
shows that $G$ has a subgroup $H$ that acts transitively on $m$ 
of the roots of $f(x)-\alpha$, and trivially on the remaining roots.
Thus, $G$ satisfies the hypotheses of 
Lemma~\ref{lem:genSd}, and hence $G\cong S_d$.
\end{proof}

\begin{proposition}
\label{prop:Gtrans}
Let $d$, $m$,  $x_0=s/t$, $b=u/w$, and $f$ be as in
Theorem~\ref{thm:biggalois} or Theorem~\ref{thm:biggalois23}.
Fix $n\geq 1$, and let  $\alpha_1, \dots, \alpha_{d^{n-1}}$ denote the roots of $f^{n-1}(x)-x_0$.
For each $i=1,\ldots, d^{n-1}$, let $M_i:=K(f^{-1}(\alpha_i))$
and $\widehat{M_i}:=\prod_{j\neq i}M_j$.
Then $\Gal(K(f^{-n}(x_0)) / \widehat{M_i})$
contains an element that acts as a single transposition
on the elements of $f^{-1}(\alpha_i)$.
\end{proposition}

\begin{proof}
Fix $i\in \{1,\ldots, d^{n-1}\}$.
Let $\fp$ be the prime described in condition~(3) of Theorem~\ref{thm:biggalois}
for $\Delta(f^n(x)-x_0)$.
As in the proof of Proposition~\ref{prop:galSd},
there must be a prime $\fP$ of $K(\alpha_i)$ lying above $\fp$
that divides $\Delta(f(x)-\alpha_i)$ to an odd power.
By Lemma~\ref{lemma.disc} applied to $\Delta(f(x)-\alpha_i)$,
the prime $\fP$ must ramify in $M_i=K(f^{-1}(\alpha_i))$;
and by Lemma~\ref{lemma.transp}, the inertia group of $\fP$ in
$\Gal(M_i/K(\alpha_i))$ is generated by an element that acts as a transposition
on $f^{-1}(\alpha_i)$.

In addition, because $\fP$ lies over $\fp$, with $\fp\nmid \Delta(f^{n-1}(x)-x_0)$,
the prime $\fP$ does not ramify in $K_{n-1}:=K(f^{-(n-1)}(x_0))$. 
It suffices to show that $\fP$ does not ramify in $\widehat{M}_i$.
Indeed, in that case, the inertia group of $\fP$ in $\Gal(\widehat{M}_i / K(\alpha_i))$
is trivial. 
	
Suppose that there is some $j\neq i$ and some prime $\fQ$ of $K_{n-1}$
with $\fQ | \fP$ and which ramifies in $K_{n-1} M_j$. Then $\alpha_j$
is a critical value of $f$ modulo $\fQ$; but so is $\alpha_i$,
since $\fP$ ramifies in $M_i$.
If either $\alpha_i$ or $\alpha_j$ is congruent to $0$ or $\infty$ modulo~$\fQ$,
then $x_0=f^{n-1}(\alpha_i)=f^{n-1}(\alpha_j)$ must also be congruent 
to $0$ or $\infty$ modulo~$\fQ$, and hence $\fQ\mid stuw$, contradicting
the assumption that $\fp\nmid stuw$. Thus, $\alpha_i$ and $\alpha_j$
must be of the form $f(\eta)$ and $f(\xi)$, respectively,
where $\eta,\xi$ satisfy
\[ \eta^{d-m}\equiv \xi^{d-m}\equiv \frac{mb}{d}\pmod{\fQ} . \]
Therefore, as in the proof of Lemma~\ref{lemma.transp}, there is a
$(d-m)$-th root of unity $\zeta$ so that
$\eta\equiv \zeta\xi \pmod{\fQ}$, and hence
\[ \alpha_i \equiv f(\eta) \equiv \zeta^m f(\xi) \equiv \zeta^m \alpha_j \pmod{\fQ}. \]
Applying $f^{n-1}$, we have
\[ x_0 = f^{n-1}(\alpha_i) \equiv \zeta^{m^n} f^{n-1}(\alpha_j)
= \zeta^{m^n} x_0 \pmod{\fQ} .\]
Since $\fp\nmid (d-m)$, it follows that $\zeta^{m^n}\equiv 1 \pmod{\fQ}$.
Therefore, because $(m,d-m)=1$, we have $\zeta\equiv 1 \pmod{\fQ}$,
and hence $\alpha_i\equiv\alpha_j \pmod{\fQ}$.
But in that case, $f^{n-1}(x)-x_0$ has  multiple roots modulo $\fQ$,
yielding
\[ \fp | \Delta(f^{n-1}(x)-x_0), \]
which is a contradiction.
Thus, $\fP$ does not ramify in $K_{n-1} M_j$ for any $j\neq i$.
Taking the compositum, $\fP$ does not ramify in $\widehat{M_i}$.
\end{proof}

\begin{proof}[Proof of Theorems~\ref{thm:biggalois} and~\ref{thm:biggalois23}]
We proceed by induction on $n$. The conclusion is trivial for $n=0$.
Assuming it holds for $n-1$, we have in particular
that $f^{n-1}(x)-x_0$ is irreducible over $K$,
with roots $\alpha_1, \dots, \alpha_{d^{n-1}}$.
For each $i=1,\ldots, d^{n-1}$, we claim that $\Gal(K_n/\widehat{M_i})\cong S_d$,
where $K_n:=K(f^{-n}(x_0))$, and $\widehat{M_i}$ is
as in Proposition~\ref{prop:Gtrans}.
	
To prove the claim, let $M_i$ be as in Proposition~\ref{prop:Gtrans},
and note that
\begin{equation}
\label{eq:galisom}
\Gal(K_n/\widehat{M_i})\cong \Gal(M_i/\widehat{M_i}\cap M_i),
\end{equation}
where the isomorphism is not just of abstract groups, but of subgroups of $S_d$
acting on $f^{-1}(\alpha_i)$.

Since $M_i$ and $\widehat{M_i}$ are both Galois extensions of $K(\alpha_i)$,
their subfield $M_i\cap \widehat{M_i}$ is also Galois over $K(\alpha_i)$.
Hence, $\Gal(M_i/\widehat{M_i}\cap M_i)$
is a normal subgroup of $\Gal(M_i/K(\alpha_i))$,
which is isomorphic to $S_d$, by Proposition~\ref{prop:galSd}.
On the other hand, by Proposition~\ref{prop:Gtrans},
the isomorphic group $\Gal(K_n/\widehat{M_i})$ contains a transposition.
By equation~\eqref{eq:galisom},
$\Gal(K_n/\widehat{M_i})$
is a normal subgroup of $S_d$ that contains a transposition,
and therefore it is all of $S_d$, as claimed.
	
Thus, for each $i=1,\ldots, d^{n-1}$, we see that $\Gal(K_n/K_{n-1})$
contains a subgroup $H_i$ isomorphic to $S_d$
and which acts trivially on $f^{-1}(\alpha_j)$ for each $j\neq i$.
It follows that $\Gal(K_n/K_{n-1})$ contains a subgroup $H:=\prod_i H_i$ of order $(d!)^{d^{n-1}}$.
Hence, $\Gal(K_n/K)$ is isomorphic to a subgroup of $[S_d]^n$ of order at least
\[ \big| \Gal(K_{n-1}/K) \big| \cdot | H |
= \big| [S_d]^{d^{n-1}} \big| (d!)^{d^{n-1}}
= (d!)^{1 + d + \cdots + d^{n-2} + d^{n-1}} = \big| [S_d]^n \big|. \]
Therefore, $\Gal(K_n/K) \cong [S_d]^n$, as desired.
\end{proof}

\section{Proof of Odoni's Conjecture for $d$ even}\label{section:even}

We now prove Theorem~\ref{thm:maintheorem} for even degree $d$:

\begin{theorem}
\label{thm:monicevenodoni}
Let $K$ be a number field, and let $d\geq 2$ be an even integer.
Then there is a monic polynomial $f(x)\in K[x]$ of degree $d$
and a rational point $x_0\in K$ such that for all $n\geq 0$,
\[ \Gal \Big( K\big( f^{-n}(x_0) \big) / K \Big) \cong [S_d]^n .\]
\end{theorem}

Before proving Theorem~\ref{thm:monicevenodoni}, we need one more lemma.

\begin{lemma}
\label{lem:newpmonic}
Let $d\geq 2$ be an integer, let $K$ be a number field with ring of integers $\cO_K$,
and let $s,t \in \cO_K$. Suppose that $s(d-1)$ and $dt(s^{d-1} + t^{d-1})$
are relatively prime.
Let $x_0=s/t$ and
\[ f(x) = x^d - b x^{d-1} \in K[x], \quad\text{where }
b= \frac{x_0^d}{x_0^{d-1} + 1} = \frac{s^d}{t(s^{d-1}+t^{d-1})} .\]
Let $\eta = (d-1)b/d$ be the unique nonzero critical point of $f$.
Then for every $n\geq 1$,
\begin{equation}
\label{eq:newfactormon}
F_n := s^{-1} \Big( d t \big( s^{d-1} + t^{d-1} \big)\Big)^{d^n}  \big[ f^n(\eta) - x_0 \big]
\end{equation}
is an $\cO_K$-integer relatively prime to $d(d-1)st(s^{d-1}+t^{d-1})$.
Moreover, if $uF_n$ is not a square in $K$ for any unit $u\in\cO_K^{\times}$,
then there is a prime $\fq\nmid d(d-1)st(s^{d-1}+t^{d-1})$ of $\cO_K$ dividing
$\Delta(f^n(x)-x_0)$ to an odd power, and such that
$v_{\fq}\big(\Delta( f^{\ell}(x)-x_0)\big)=0$ for all $0\leq \ell < n$.
\end{lemma}

\begin{proof}
Let $D=s^{d-1} + t^{d-1}$.
We claim that for any $n\geq 1$,
\begin{equation}
\label{eq:inducfn}
(dtD)^{d^n}f^n(\eta) = s^{de_n} (d-1)^{(d-1)^n} M_n
\end{equation}
where $M_n\in\cO_K$ is relatively prime to $d(d-1)stD$,
and where
\[ e_n=(d-1)^n + (d-1)^{n-1} + \cdots + (d-1) + 1. \]

Proceeding by induction on $n$, a direct computation shows
\begin{equation}
\label{eq:f1eta}
f(\eta) = - (d-1)^{d-1} \Big( \frac{b}{d} \Big)^d
= \frac{ - (d-1)^{d-1} s^{d^2} }{\big(dtD\big)^d},
\end{equation}
proving the claim for $n=1$, with $M_1=-1$.
Given the claim for a particular $n\geq 1$, we have
\[ \frac{(dtD)^{d^{n+1}} f^{n+1}(\eta)}{(d-1)^{(d-1)^{n+1}} s^{de_{n+1}} M_n^{d-1}}
= (d-1)^{(d-1)^{n}} s^{d(e_n - 1)} M_n - d^{d^n}(tD)^{d^n - 1}, \]
since $e_{n+1} = (d-1)e_n + 1$.
Setting
\[ M_{n+1}:= M_n^{d-1}\Big( (d-1)^{(d-1)^{n}} s^{d(e_n - 1)} M_n - d^{d^n}(tD)^{d^n - 1} \Big)
\in \cO_K, \]
we see that $M_{n+1}$ is relatively prime to $d(d-1)stD$, proving the claim.

Fix $n\geq 1$. It is immediate from equations~\eqref{eq:newfactormon}
and~\eqref{eq:inducfn} that
\[ F_n = s^{de_n-1} (d-1)^{(d-1)^n} M_n  - d^{d^n} ( t D)^{d^n-1} \in\cO_K, \]
which is relatively prime to $d(d-1)stD$, as desired.

For the remainder of the proof, assume that $uF_n$ is not a square in $\cO_K$
for any unit $u\in\cO_K^{\times}$.
Then there is a prime $\fq$ of $\cO_K$ dividing $F_n$ to an odd power.
We must have $\fq \nmid d(d-1)stD$.

The factor consisting of
the product over critical points in discriminant formula~\eqref{eq:iterdisc}
for $\Delta(f^n(x)-x_0)$ is
\[ (-x_0)^{d-2} (f^n(\eta) - x_0) =
\frac{-s^{d-1} F_n}{ t^{d^n + d -2} (d D)^{d^n}} . \]
Thus, $\fq$ divides this factor to an odd power.
By formula~\eqref{eq:iterdisc}, then, it suffices to show that 
\begin{equation}
\label{eq:nopDeltamon}
v_{\fq}\Big( \Delta\big( f^{\ell}(x) - c \big) \Big) = 0
\quad\text{for all} \quad 0\leq \ell \leq n-1 .
\end{equation}
Suppose not. Let $0\leq \ell \leq n-1$ be the smallest index for which
equation~\eqref{eq:nopDeltamon} fails. Then by formula~\eqref{eq:iterdisc} again,
we must have
\[ v_{\fq} \big( f^{\ell}(\eta) - x_0 \big) \geq 1,
\quad\text{i.e.,} \quad
f^{\ell}(\eta) \equiv x_0 \pmod{\fq} . \]
Thus,
\[ f^{\ell + 1}(\eta) \equiv b \pmod{\fq},
\quad\text{and}\quad
f^{\ell + j}(\eta) \equiv 0 \pmod{\fq}
\quad \text{for all} \quad j\geq 2. \]
Therefore, $f^n(\eta)$ is congruent to either $b$ or $0$ modulo $\fq$.
However, $f^n(\eta)\equiv x_0 \pmod{\fq}$, and $x_0\not\equiv b,0\pmod{\fq}$,
since $x_0-b = -st^{d-2}/D$ and $\fq\nmid stD$.
This contradiction proves equation~\eqref{eq:nopDeltamon} and hence the Lemma.
\end{proof}

\begin{proof}[Proof of Theorem~\ref{thm:monicevenodoni}]
Fix $d\geq 2$ even. 
It suffices to show that there is some $f(x)\in K[z]$ satisfying the hypotheses
of Theorem~\ref{thm:biggalois} or~\ref{thm:biggalois23}, with $\deg(f)=d$.

\textbf{Step~1}. We will show that there is a prime $\fp$ of $\cO_K$
such that all units $u\in\cO_K^{\times}$ are squares modulo $\fp$,
and so is $1-d$, with $\fp\nmid d(d-1)$.
To do so, let $u_1,\ldots,u_r$ be generators of the 
(finitely-generated) unit group $\cO_K^{\times}$. It suffices to find a prime $\fp\nmid d(d-1)$
for which each of $1-d,u_1,\ldots, u_r$ is a square modulo $\fp$.

Let $L=K(\sqrt{1-d},\sqrt{u_1},\ldots,\sqrt{u_r})$, which is a Galois extension of $K$.
By the Chebotarev Density Theorem, there are infinitely many primes $\fp$ of $\cO_K$
at which Frobenius acts trivially on $L$ modulo $\fp$. Choosing any such prime $\fp$
that does not ramify in $L$ and does not divide $d(d-1)$,
it follows that each of $1-d,u_1,\ldots,u_r$ have square roots modulo $\fp$, as desired.

\textbf{Step~2}.
Let $\fp_1$ be a prime of $\cO_K$ not dividing $d(d-1)\fp$.
Choose $s_0\in \cO_K$ that is not a square modulo $\fp$, and choose
$s_1\in \cO_K$ with $v_{\fp_1}(s_1)=1$.
Since the three ideals $d(d-1)$, $\fp$, and $\fp_1^2$ are pairwise relatively prime,
the Chinese Remainder Theorem shows that there is some $s\in\cO_K$ such that
\[ s\equiv 1 \pmod{d(d-1)},
\quad s\equiv s_0 \pmod{\fp},
\quad\text{and}\quad s\equiv s_1 \pmod{\fp_1^2}. \]
Thus, $s$ is not a square modulo $\fp$, and $v_{\fp_1}(s)=1$.

Choose $t_0\in\cO_K$ with $v_{\fp}(t_0)=1$.
Since the ideals $s(d-1)$ and $\fp^2$ are relatively prime,
the Chinese Remainder Theorem shows that there is some $t\in\cO_K$ such that
\[ t\equiv 1 \pmod{s(d-1)},
\quad\text{and}\quad t\equiv t_0 - s \pmod{\fp^2}. \]
In particular,
\begin{align*}
s^{d-1} + t^{d-1} & \equiv s^{d-1} + (t_0 - s)^{d-1}
\equiv s^{d-1} + \big(-s^{d-1} + (d-1) s^{d-2}t_0\big)
\\
&\equiv (d-1) s^{d-2} t_0 \pmod{\fp^2},
\end{align*}
and therefore $v_{\fp}(s^{d-1} + t^{d-1}) = v_{\fp}(t_0)= 1$.
In addition,
\[ s^{d-1} + t^{d-1} \equiv 2 \pmod{d-1}, \]
and hence $d-1$ and $s^{d-1}+t^{d-1}$ are relatively prime,
since $d-1$ is odd.

Let $\fp_2=\fp$.
By our choices above, note that $s$ and $dt$ are also relatively prime,
and so are $t$ and $d-1$. Thus,
\[ s(d-1) \quad\text{and}\quad dt(s^{d-1} + t^{d-1})
\quad\text{are relatively prime} \]
as elements of $\cO_K$. Hence,
setting $x_0=s/t$ and $b=s^d/(t(s^{d-1} + t^{d-1}))$, we have
\begin{equation}
\label{eq:bx0vals}
v_{\fp_1}(x_0)=1, \quad v_{\fp_1}(b) = d, \quad
v_{\fp_2}(x_0) = 0, \quad\text{and}\quad v_{\fp_2}(b) = -1.
\end{equation}
Define $f\in K[x]$ by $f(x) = x^d - bx^{d-1}$.

\textbf{Step 3}.
Let $m=d-1$.
If $d=2$, then $f$ is of the form of Theorem~\ref{thm:biggalois23},
and the relations~\eqref{eq:bx0vals}
show that $f$ satisfies condition~(1) of Theorem~\ref{thm:biggalois}.
Similarly, if $d\geq 4$, then
$f$ is of the form of Theorem~\ref{thm:biggalois},
and relations~\eqref{eq:bx0vals}
show that $f$ satisfies conditions~(1) and~(2) of that Theorem.
In both cases,
we claim that for each $n\geq 1$, the quantity $F_n$ of
equation~\eqref{eq:newfactormon} is not a square modulo $\fp=\fp_2$.
Since all units of $\cO_K$ are squares modulo $\fp$,
Lemma~\ref{lem:newpmonic} will then guarantee that
condition~(3) of Theorem~\ref{thm:biggalois} also holds,
yielding the desired result.
Thus, it suffices to prove our claim: $F_n$ is not a square
modulo $\fp$, for every $n\geq 1$.

Write $D=s^{d-1}+t^{d-1}$, so that $b=s^d/(tD)$, and $v_{\fp}(D)=1$.
Let $\eta=(d-1)b/d$, which is the only critical point of $f$ besides $0$ and $\infty$.
Then for all $n\geq 1$, observe that
\begin{equation}
\label{eq:fnetamod}
(dtD)^{d^n} f^n(\eta) \equiv \Big( -(d-1)^{d-1} s^{d^2} \Big)^{d^{n-1}} \pmod{dtD}.
\end{equation}
Indeed, for $n=1$, equation~\eqref{eq:fnetamod} is immediate
from equation~\eqref{eq:f1eta}.
For $n\geq 2$, we have
\[ (dtD)^{d\ell} f\bigg( \frac{y}{(dtD)^{\ell}} \bigg) \equiv y^d \pmod{dtD}\]
for $y\in\cO_K$ relatively prime to $dtD$ and $\ell\geq 2$,
and hence equation~\eqref{eq:fnetamod} follows by induction on $n$.

Equation~\eqref{eq:fnetamod} shows that
$(dtD)^{d^n} f^n(\eta)$ is a nonzero square modulo $\fp$ for every $n\geq 1$.
(For $n=1$, recall that $1-d$ is a square modulo $\fp$.)
In addition, $(dtD)^{d^n} x_0\equiv 0 \pmod{\fp}$. Thus,
from the definition of $F_n$ in equation~\eqref{eq:newfactormon},
we see that $sF_n$ is a nonzero square modulo $\fp$.
Since $s$ is not a square modulo $\fp$, we have proven our claim
and hence the Theorem.
\end{proof}

\section{Proof of Odoni's Conjecture for $d$ odd, for most $K$}\label{section:odd}

We now prove Theorem~\ref{thm:maintheorem} for odd degree $d$:

\begin{theorem}
\label{thm:monicoddodoni}
Let $d\geq 3$ be an odd integer, and
let $K$ be a number field in which $d$ and $d-2$ are not both squares.
Then there is a monic polynomial $f(x)\in K[x]$ of degree $d$
and a rational point $x_0\in K$ such that for all $n\geq 0$,
\[ \Gal \Big( K\big( f^{-n}(x_0) \big) / K \Big) \cong [S_d]^n .\]
\end{theorem}

Note that if $[K:\bQ]$ is odd, then Theorem~\ref{thm:monicoddodoni} yields
Odoni's conjecture in all odd degrees $d\geq 3$. After all, for $d\geq 3$ odd,
at least one of $d$ and $d-2$ is not a square in $\bQ$. Thus, 
any number field of odd degree over $\bQ$ cannot contain square roots
of both $d$ and $d-2$.

Before proving Theorem~\ref{thm:monicoddodoni}, we need one more lemma.

\begin{lemma}
\label{lem:newpmonicodd}
Let $d\geq 3$ be an odd integer, let $K$ be a number field with ring of integers $\cO_K$,
and let $s,t \in \cO_K$.
Suppose that $2(d-2)s$ and $dt$ are relatively prime.
Let $x_0=s/t$ and
\[ f(x) = x^d - b x^{d-2} \in\bQ[x], \quad\text{where }
b= x_0^2 = \frac{s^2}{t^2}. \]
Let $\pm\eta = \pm x_0 \sqrt{(d-2)/d}$ be the two nonzero critical points of $f$.
Then for every $n\geq 1$,
\begin{equation}
\label{eq:newfactormonicodd}
F_n:= s^{-2}\big( dt^2 \big)^{d^n} \big[ f^n(\eta)^2 - x_0^2 \big]
\end{equation}
is an $\cO_K$-integer relatively prime to $2d(d-2)st$.
Moreover, if $uF_n$ is not a square in $K$ for any unit $u\in\cO_K^{\times}$,
then there is a prime $\fq\nmid 2d(d-2)st$ of $\cO_K$ dividing
$\Delta(f^n(x)-x_0)$ to an odd power, and such that
$v_{\fq}\big(\Delta( f^{\ell}(x)-x_0)\big)=0$ for all $0\leq \ell < n$.
\end{lemma}

\begin{proof}
We claim that for any $n\geq 1$,
\begin{equation}
\label{eq:inducfnodd}
(dt^2)^{d^n}f^n(\eta)^2 = 4^{(d-2)^{n-1}} (d-2)^{(d-2)^n}s^{2e_n} M_n^2
\end{equation}
where $M_n\in\cO_K$ is relatively prime to $2d(d-2)st$,
and where
\[ e_n=(d-2)^n + 2(d-2)^{n-1} + \cdots + 2(d-2) + 2 .\]

Observe that $f(\sqrt{x})^2 = x^{d-2}(x-b)^2$. Thus, we have
$f^n(\eta)^2 = g^n(\eta^2)$, where $g(x)=x^{d-2}(x-b)^2$.
Since $\eta^2=(d-2)x_0^2/d$, a direct computation shows
\begin{equation}
\label{eq:f1etaodd}
f(\eta)^2 = g(\eta^2) = \frac{\big((d-2)s^2\big)^{d-2}}{(dt^2)^{d-2}}
\bigg( \frac{(d-2)s^2}{dt^2} - \frac{s^2}{t^2} \bigg)
= \frac{ 4 (d-2)^{d-2} s^{2d} }{(dt^2)^d},
\end{equation}
proving the claim for $n=1$, with $M_1=1$.
Assuming equation~\eqref{eq:inducfnodd} holds
for a particular $n\geq 1$, we have
\[ \frac{(dt^2)^{d^{n+1}} f^{n+1}(\eta)}{4^{(d-2)^n} (d-2)^{(d-2)^{n+1}} s^{2e_{n+1}} M_n^{2(d-2)}}
= \Big( 4^{(d-2)^{n-1}} (d-2)^{(d-2)^n} s^{2e_n-2} M_n^2
- d^{d^n} t^{2(d^n-1)} \Big)^2 , \]
since $e_{n+1} = (d-2)e_n + 2$.
Setting
\[ M_{n+1} := M_n^{d-2}\Big( 4^{(d-2)^{n-1}} (d-2)^{(d-2)^n} s^{2e_n-2} M_n^2
- d^{d^n} t^{2(d^n-1)} \Big) \in \cO_K, \]
we see that $M_{n+1}$ is relatively prime to $2std(d-2)$, proving the claim.

Fix $n\geq 1$. It is immediate from equations~\eqref{eq:newfactormonicodd}
and~\eqref{eq:inducfnodd} that
\[ F_n = 4^{(d-2)^{n-1}} (d-2)^{(d-2)^n} s^{2e_n-2} M_n^2
- d^{d^n} t^{2d^n-2} \in\cO_K, \]
which is relatively prime to $2d(d-2)st$, as desired.

For the remainder of the proof, assume that $uF_n$ is not a square in $\cO_K$
for any unit $u\in\cO_K^{\times}$.
Then there is a prime $\fq$ of $\cO_K$ dividing $F_n$ to an odd power.
We must have $\fq \nmid 2d(d-2)st$.

The factor consisting of
the product over critical points in discriminant formula~\eqref{eq:iterdisc}
for $\Delta(f^n(x)-x_0)$ is
\[ -x_0^{d-3} (f^n(\eta)^2 - x_0^2) =
\frac{-s^{d-1} F_n}{d^{d^n} t^{2d^n + d -3}} . \]
Thus, $\fq$ divides this factor to an odd power.
By formula~\eqref{eq:iterdisc}, then, it suffices to show that 
\begin{equation}
\label{eq:nopDeltamonodd}
v_{\fq}\Big( \Delta\big( f^{\ell}(x) - c \big) \Big) = 0
\quad\text{for all} \quad 0\leq \ell \leq n-1 .
\end{equation}
Suppose not. Let $0\leq \ell \leq n-1$ be the smallest index for which
equation~\eqref{eq:nopDeltamonodd} fails. Then by formula~\eqref{eq:iterdisc} again,
we must have
\[ v_{\fq} \big( f^{\ell}(\eta)^2 - x_0^2 \big) \geq 1,
\quad\text{i.e.,} \quad
f^{\ell}(\eta) \equiv x_0  \text{ or } f^{\ell}(-\eta) \equiv x_0  \pmod{\fq} . \]
Without loss, assume $f^{\ell}(\eta) \equiv x_0  \pmod{\fq}$.
Then $f^{\ell + j}(\eta) \equiv 0 \pmod{\fq}$ for all $j\geq 1$.
In particular, $f^n(\eta) \equiv 0 \pmod{\fq}$.
However,
\[ f^n(\eta)\equiv x_0 \not \equiv 0 \pmod{\fq}. \]
This contradiction proves equation~\eqref{eq:nopDeltamonodd} and hence the Lemma.
\end{proof}

\begin{proof}[Proof of Theorem~\ref{thm:monicoddodoni}]
Fix $d\geq 3$ odd.
It suffices to show that there is some $f(x)\in K[z]$ satisfying the hypotheses
of Theorem~\ref{thm:biggalois} or~\ref{thm:biggalois23}, with $\deg(f)=d$.

\textbf{Case 1: $d$ is not a square in $K$; Step 1}.
We first show that there is a prime $\fp$ of $\cO_K$
such that all units $u\in\cO_K^{\times}$ are squares modulo $\fp$,
with $\fp\nmid 2d(d-2)$,
and such that $d$ is not a square modulo $\fp$,
by adjusting the method of Step~1 of the proof of Theorem~\ref{thm:monicevenodoni}
with inspiration from well-known argument of Hall \cite{Hall}.

Let $\fq$ be a prime of $\cO_K$ dividing $d$ to an odd power;
note that $\fq$ lies above an odd prime of $\ZZ$.
Let $u_1,\ldots,u_r$ be generators of the 
(finitely-generated) unit group $\cO_K^{\times}$,
and let $L=K(\sqrt{u_1},\ldots,\sqrt{u_r})$. The discriminant $\Delta(L/K)$
must divide a power of $2$, and hence $\fq$ cannot
ramify in $L$. However, $\fq$ ramifies in $K(\sqrt{d})$, and therefore
$K(\sqrt{d})\cap L = K$. It follows that $L(\sqrt{d})/L$ is a quadratic extension.

Therefore, there exists
$\sigma\in\Gal(L(\sqrt{d})/L)\subseteq\Gal(L(\sqrt{d})/K)$
that fixes $\sqrt{u_1},\ldots,\sqrt{u_r}$ but has $\sigma(\sqrt{d})=-\sqrt{d}$.
By the Chebotarev Density Theorem, there are infinitely many primes $\fp$ of $\cO_K$
at which Frobenius acts trivially on $L$ modulo $\fp$, but nontrivially on $\sqrt{d}$.
Choosing any such prime $\fp_1$
that does not ramify in $L$ and does not divide $2d(d-2)$,
it follows that all units of $\cO_K$ are squares modulo $\fp_1$, but $d$ is not.

\textbf{Case 1, Step~2}.
Let $\fp_1$ be the prime of $\cO_K$ found in Step~1,
and choose $s\in\cO_K$ relatively prime to $d$ such that $v_{\fp_1}(s)=1$.
Let $\fp_2$ be a prime of $\cO_K$ not dividing $2d(d-2)s$,
and choose $t\in\cO_K$ relatively prime to $2(d-2)s$ such that $v_{\fp_2}(t)=1$.
Let $x_0=s/t$ and $b=x_0^2$, so that
\begin{equation}
\label{eq:bx0valsodd}
v_{\fp_1}(x_0)=1, \quad v_{\fp_1}(b) = 2, \quad
v_{\fp_2}(x_0) = -1, \quad\text{and}\quad v_{\fp_2}(b) = -2.
\end{equation}
Define $f\in K[x]$ by $f(x) = x^d - bx^{d-1}$.

\textbf{Case 1, Step 3}.
Let $m=d-2$. If $d=3$, then $f$ is of the form of Theorem~\ref{thm:biggalois23},
and the relations~\eqref{eq:bx0valsodd}
show that $f$ satisfies condition~(1) of Theorem~\ref{thm:biggalois}.
Similarly, if $d\geq 5$, then $f$ is of the form of Theorem~\ref{thm:biggalois},
and relations~\eqref{eq:bx0valsodd}
show that $f$ satisfies conditions~(1) and~(2) of that Theorem.
In both cases,
we claim that for each $n\geq 1$, the quantity $F_n$ of
equation~\eqref{eq:newfactormonicodd} is not a square modulo $\fp=\fp_1$.
Since all units of $\cO_K$ are squares modulo $\fp$,
Lemma~\ref{lem:newpmonicodd} will then guarantee that
condition~(3) of Theorem~\ref{thm:biggalois} also holds,
yielding the desired result.
Thus, it suffices to prove our claim: $F_n$ is not a square
modulo $\fp$, for every $n\geq 1$.

Let $\eta=x_0\sqrt{(d-2)/d}$, so that $\pm\eta$
are the only two critical points of $f$ besides $0$ and $\infty$.
Then for all $n\geq 1$, observe that
\begin{equation}
\label{eq:fnetamod1}
s^{-2} f^n(\eta)^2 \equiv 0 \pmod{s},
\end{equation}
since equation~\eqref{eq:inducfnodd} shows $f^n(\eta)^2$ is divisible
by a higher power of $s$.

Combining the definition of $F_n$ in equation~\eqref{eq:newfactormonicodd}
and equation~\eqref{eq:fnetamod1}, we have
\[ F_n \equiv -d^{d^n} t^{2d^n - 2} \pmod{\fp}, \]
since $\fp=\fp_1 | s$.
Since $d$ is odd and not a square modulo $\fp$,
we have proven our claim
and hence Case~1 of the Theorem.

\textbf{Case 2: $d-2$ is not a square in $K$; Step 1}.
By the same argument as in Case~1, Step~1, but this
time with the odd number $d-2$ in place of $d$,
there is a prime $\fp$ of $\cO_K$
such that all units $u\in\cO_K^{\times}$ are squares modulo $\fp$,
with $\fp\nmid 2d(d-2)$,
and such that $d-2$ is not a square modulo $\fp$.

\textbf{Case 2, Step~2}.
Let $\fp_2$ be the prime $\fp$ of $\cO_K$ found in Step~1,
and choose $t\in\cO_K$ relatively prime to $2(d-2)$ such that $v_{\fp_2}(t)=1$.
Let $\fp_1$ be a prime of $\cO_K$ not dividing $2d(d-2)t$,
and choose $s\in\cO_K$ relatively prime to $dt$ such that $v_{\fp_1}(s)=1$.
Let $x_0=s/t$ and $b=x_0^2$, and define
$f\in K[x]$ by $f(x) = x^d - bx^{d-2}$.

\textbf{Case 2, Step 3}.
Since $d-2$ is not a square, we have $d\geq 5$.
As in Case~1, relations~\eqref{eq:bx0valsodd} hold,
and they
show that $f$ satisfies conditions~(1) and~(2) of Theorem~\ref{thm:biggalois},
with $m=d-2$, and hence $d-m=2$.
Also as in Case~1, it then suffices to prove the following claim:
that for each $n\geq 1$,
the quantity $F_n$ is not a square modulo $\fp=\fp_2$.

As before, let $\eta=x_0\sqrt{(d-2)/d}$.
Then for all $n\geq 1$, observe that
\begin{equation}
\label{eq:fnetamod2}
(dt^2)^{d^n} f^n(\eta)^2 \equiv \Big( 4(d-2)^{d-2} s^{2d} \Big)^{d^{n-1}} \pmod{dt^2}.
\end{equation}
Indeed, for $n=1$, equation~\eqref{eq:fnetamod2} is immediate
from equation~\eqref{eq:f1etaodd}.
For $n\geq 2$, we have
\[ (dt^2)^{d\ell} f\bigg( \frac{\sqrt{y}}{(dt^2)^{\ell/2}} \bigg)^2 \equiv y^d \pmod{dt^2}\]
for $y\in\cO_K$ relatively prime to $dt$ and for $\ell\geq 2$.
Hence, equation~\eqref{eq:fnetamod2} follows by induction on $n$.

Since $d$ is odd and $d-2$ is not a square modulo $\fp$,
equation~\eqref{eq:fnetamod2} shows that the quantity
$s^{-2}(dt^2)^{d^n} f^n(\eta)^2$ is a nonsquare modulo $\fp$ for every $n\geq 1$.
In addition, $(dt^2)^{d^n} x_0\equiv 0 \pmod{\fp}$, since $\fp=\fp_2 | t$.
Thus, from the definition of $F_n$ in equation~\eqref{eq:newfactormonicodd},
we have proven our claim and hence the Theorem.
\end{proof}

\textbf{Acknowledgments}: The first author gratefully acknowledges the support of NSF grant DMS-1501766.

\bibliographystyle{amsalpha}
\bibliography{biblio}

\end{document}